\documentclass[11pt]{amsart}
\usepackage[colorlinks, citecolor=blue, pagebackref]{hyperref}

\usepackage{amsmath}

\usepackage{extarrows}

\newtheorem{theorem}{Theorem}[section]
\newtheorem{lemma}[theorem]{Lemma}

\theoremstyle{definition}
\newtheorem{definition}[theorem]{Definition}
\newtheorem{question}[theorem]{Question}

\newtheorem{proposition}[theorem]{Proposition}
\newtheorem{corollary}[theorem]{Corollary}
\newtheorem{remark}[theorem]{Remark}

\theoremstyle{remark}

\newcommand{\be}{\begin{equation}}
\newcommand{\ee}{\end{equation}}

\numberwithin{equation}{section}

\begin{document}
\title[On the spectral rigidity]{The spectral rigidity of Ricci soliton and Einstein-type manifolds}
\author{Ping Li}
\address{School of Mathematical Sciences, Fudan University, Shanghai 200433, China}

\email{pinglimath@fudan.edu.cn}

\author{Xiaomei Sun}

\address{College of information, Huazhong Agricultural University, Wuhan 430070, China}

\email{xmsunn@mail.hzau.edu.cn}

\author{Anqiang Zhu}

\address{School of Mathematics and Statistics, Wuhan University, Wuhan 430072, China}

\email{aqzhu.math@whu.edu.cn}

\thanks{This work is partially supported by the National
Natural Science Foundation of China (Grant No. 12371066).}




\begin{abstract}
We are concerned in this article with a classical topic in spectral geometry dating back to McKean-Singer, Patodi and Tanno: whether or not the constancy of sectional curvature (resp. holomorphic sectional curvature) of a compact Riemannian manifold (resp. K\"{a}hler manifold) can be completely determined by the eigenvalues of its $p$-Laplacian for a \emph{single} integer $p$? We treat this question under two conditions: gradient shrinking Ricci soliton for Riemannian manifolds and cohomologically Einstein for K\"{a}hler manifolds. We show that, with some sporadic unknown cases, this is true for each $p$. Furthermore, we show that the condition of being isospectral can be relaxed to a suitable almost-isospectral version.
\end{abstract}

\maketitle


\section{Introduction}\label{introduction}
Unless otherwise stated, all the Riemannian (resp. K\"{a}hler) manifolds in this article are closed, connected and oriented (resp. with the canonical orientation).

Let $(M,g)$ be an $m$-dimensional Riemannian manifold, $\Omega^p(M)$ ($0\leq p\leq m$) the space of smooth exterior $p$-forms on $M$, $d:~\Omega^p(M)\rightarrow\Omega^{p+1}(M)$ the operator of exterior differentiation, and $d^{\ast}:~\Omega^p(M)\rightarrow\Omega^{p-1}(M)$ the formal adjoint of $d$ relative to the Riemannian metric $g$. Here $\Omega^p(M):=0$ provided that $p=-1$ or $m+1$. We have, for each $0\leq p\leq m$, the following second-order self-adjoint elliptic operator, the Laplacian acting on $p$-forms:
\be\label{delta}\Delta_p:=dd^{\ast}+d^{\ast}d:~\Omega^p(M)
\longrightarrow\Omega^p(M).\ee
It is well-known that $\Delta_p$ has an infinite discrete sequence
\be0\label{eigenvalue}\leq\lambda_{1,p}\leq\lambda_{2,p}\leq\cdots\leq\lambda_{i,p}\leq
\cdots\uparrow+\infty\ee
of eigenvalues and each of them is repeated as many times as its multiplicity indicates. These $\lambda_{i,p}$ are called the \emph{spectra} of $\Delta_p$. Put
$$\text{Spec}^p(M,g):=\big\{\lambda_{1,p}, \lambda_{2,p},\ldots,\lambda_{i,p},\ldots\big\},$$
which is called the spectral set of $\Delta_p$. Duality and Hodge theory tell us that $\text{Spec}^p(M,g)=\text{Spec}^{m-p}(M,g)$ and $0\in\text{Spec}^p(M,g)$ if and only if the $p$-th Betti number $b_p(M)\neq0$ and its multiplicity is precisely $b_p(M)$.

An important problem in spectral geometry is to investigate how the geometry of $(M,g)$ can be reflected by its spectra $\{\lambda_{i,p}\}$. In general the spectra $\{\lambda_{i,p}\}$ are not able to determine a manifold up to an isometry, as Milnor has constructed in \cite{Mi} two non-isometric Riemannian structures on a $16$-dimensional manifold such that for each $p$ the spectral sets $\text{Spec}^p(\cdot)$ with respect to these Riemannian metrics are the same. Nevertheless, we may still ask to what extent the spectra $\{\lambda_{i,p}\}$ encode the geometry of $(M,g)$.

Recall that, for any positive integer $N$, the famous Minakshisundaram-Pleijel asymptotic expansion formula (\cite{MP}), which is the integration on the asymptotic expansion of the heat kernel for Laplacian, implies
\be\label{mpgformula}
\begin{split}
\text{Trace}(e^{-t\Delta_p})&=\sum_{i=0}^{\infty}\exp(-\lambda_{i,p}t)\\
&=\frac{1}{(4\pi t)^{\frac{m}{2}}}\Big[{m\choose p}\text{Vol}(M,g)+
\sum_{i=1}^{N}a_{i,p}t^i\Big]+O(t^{N-\frac{m}{2}+1}),\qquad t\downarrow0\\
&=\frac{1}{(4\pi t)^{\frac{m}{2}}}
\sum_{i=0}^{N}a_{i,p}t^i+O(t^{N-\frac{m}{2}+1}),\qquad t\downarrow0,\qquad\Big(a_{0,p}:={m\choose p}\text{Vol}(M,g)\Big).
\end{split}\ee
Here $\text{Vol}(M,g)$ is the volume of $(M,g)$ and $a_{i,p}$ ($i\geq 1$) are certain functions of the curvature, which are completely determined by the spectral set $\text{Spec}^p(M,g)$. The coefficients $a_{1,0}$ and $a_{2,0}$ were calculated by Berger and McKean-Singer (\cite{Be63}, \cite{MS}) and then in \cite[p. 277]{Pa} Patodi explicitly determined $a_{1,p}$ and $a_{2,p}$ for \emph{all} $p$. Since then this asymptotic expansion (\ref{mpgformula}) and the concrete formulas for $a_{i,p}~(i=0,1,2)$ have become indispensable tools in studying related questions in spectral geometry.

When $(M,g)$ is flat, i.e., it has constant sectional curvature $k=0$, then $a_{i,p}=0$ for all $p$ and $i\geq 1$ as these $a_{i,p}$ are functions of the curvature. McKean and Singer raised in \cite{MS} a converse question: if $a_{i,0}=0$ for all $i\geq 1$, then whether or not $(M,g)$ is flat? They proved in \cite{MS} that this is true if the dimension $m\leq 3$. Patodi further showed in \cite{Pa} that this is true if $m\leq 5$ and is false when $m>5$ by constructing counterexamples (\cite[p. 283]{Pa} or \cite[p. 65]{Pa2}). This means that in general the vanishing of $a_{i,p}$ ($i\geq 1$) for only \emph{one single} $p$ is not enough to derive the flatness. Nevertheless, applying the explicit expressions of $a_{1,p}$ and $a_{2,p}$ determined by himself in \cite{Pa}, Patodi showed that whether or not $(M,g)$ is of constant sectional curvature $k$ is completely determined by the quantities $\{a_{i,p}~|~i=0,1,2,~p=0,1\}$, i.e., by the \emph{two} spectral sets $\text{Spec}^0(M,g)$ and $\text{Spec}^1(M,g)$ (\cite[p. 281]{Pa} or \cite[p. 63]{Pa2}).

The next question is to what extent the geometry of $(M,g)$ can be recovered by a \emph{single} spectral set. Note that if two Riemannian manifolds have the same spectral set $\text{Spec}^p(\cdot)$ for a single $p$, due to $(\ref{mpgformula})$ they necessarily have the same dimension. Note also that for an $m$-dimensional Riemannian manifold we only need to consider the spectral sets $\text{Spec}^p(\cdot)$ for $p\leq[\frac{m}{2}]$ as $\text{Spec}^p(\cdot)=\text{Spec}^{m-p}(\cdot)$. In view of these two basic facts, we can pose the following question, which was initiated by Tanno (\cite{Ta1}).
\begin{question}\label{Q0}
Let $(M_1,g_1)$ and $(M_2,g_2)$ be $m$-dimensional Riemannian manifolds such that $\text{Spec}^p(M_1,g_1)=\text{Spec}^p(M_2,g_2)$
for some \emph{fixed} $p$ with $p\leq [\frac{m}{2}]$, and $(M_2,g_2)$ of constant sectional curvature $k$.  Is it true that $(M_1,g_1)$ is of constant sectional curvature $k$?
\end{question}
Question \ref{Q0} has been affirmatively verified when $(p=0, m\leq5)$, $(p=0,m=6)$ provided $k>0$ (\cite{Ta1}), $(p=1, m\in[2,3]\cup[16,93])$ (\cite{Ta2}) and
$(p=2,m\in\{6,7,14\}\cup[17,178])$ (\cite{TK}).

The notion of ``holomorphic sectional curvature" (``HSC" for short) in K\"{a}hler geometry is the counterpart of that of ``sectional curvature" in Riemannian geometry and
so it is natural to consider a similar question for K\"{a}hler manifolds, which was also initiated in \cite{Ta1}.
\begin{question}\label{Q1}
Let $(M_1,g_1,J_1)$ and $(M_2,g_2,J_2)$ be complex $n$-dimensional K\"{a}hler manifolds such that $\text{Spec}^p(M_1,g_1)=\text{Spec}^p(M_2,g_2)$
for a \emph{fixed} $p$ with $p\leq n$, and $(M_2,g_2,J_2)$ of constant HSC $c$.  Is it true that $(M_1,g_1,J_1)$ is of constant HSC $c$?
\end{question}
Recall that, up to a holomorphic isometry, $(\mathbb{C}P^n(c),g_0,J_0)$, the standard complex $n$-dimensional projective space equipped with the Fubini-Study metric with \emph{positive} constant HSC $c$, is the unique complex $n$-dimensional compact K\"{a}hler manifold with \emph{positive} constant HSC $c$ by the classical uniformization theorem. So we also have the following spectral characterization problem for $(\mathbb{C}P^n(c),g_0,J_0)$, which, to the authors' best knowledge, was first explicitly proposed by B.Y. Chen and Vanhecke in \cite{CV}.
\begin{question}\label{Q2}
Let $(M,g,J)$ be a K\"{a}hler manifold such that $\text{Spec}^p(M,g)=\text{Spec}^p(\mathbb{C}P^n(c),g_0)$
for a \emph{fixed} $p$ with $p\leq n$. Is it true that $(M,g,J)$ is holomorphically isometric to $(\mathbb{C}P^n(c),g_0,J_0)$?
\end{question}
Clearly a positive answer to Question \ref{Q1} implies that to Question \ref{Q2}.
Question \ref{Q1} was affirmatively verified when $(p=0, n\leq5)$, $(p=0,n=6)$ provided $c\neq0$ (\cite{Ta1}), and $(p=1, 8\leq n\leq 51)$ (\cite{Ta2}).
Consequently, Question \ref{Q2} is also true in these cases. When $p=2$, Question \ref{Q2} is true for all dimensions $n$, due to the work of Chen-Vanhecke (\cite{CV}), Goldberg (\cite{Go}) and the first author (\cite{Li}). In \cite{Li} we gave the most complete proof and clarified some gaps in previous literature especially when $(p,n)=(2,8)$. Moreover, for each \emph{even} $p\geq4$, we solved Question \ref{Q2} in all dimensions with at most one exception  (\cite[Theorem 1.3]{Li}). Very recently Huang-Liu-Xu-Zhi gave an explicit (finite) range of $n$ for each fixed \emph{odd} $p$ where Question \ref{Q2} holds true (\cite[Theorem 2.10]{HLXZ}) by extending some ideas from \cite{Li}. Note that in all the solved cases in Questions \ref{Q0}, \ref{Q1} and \ref{Q2}, only those positive even $p$ in Question \ref{Q2} are the known cases where they hold true for \emph{infinitely many} dimensions $n$. In particular, the $p=2$ case is the \emph{only} known one where it holds true for \emph{all} dimensions $n$.

\emph{The main purpose of this paper} is to impose an extra condition on $M_1$ in Questions \ref{Q0}-\ref{Q2} to treat them. Indeed, if $(M_1,g_1)$ in Question \ref{Q0} is further assumed to be \emph{Einstein}, it holds true for $(p=0, \text{all $m$})$ (\cite{Sa}) and $\big(p=2,m\in[2,7]\cup\{14\}\cup[17,\infty)\big)$ (\cite{TK}). The extra conditions we impose are \emph{gradient shrinking Ricci soliton} in Question \ref{Q0} and \emph{cohomologically Einstein} in Questions \ref{Q1} and \ref{Q2}. Both are natural generalizations of Einstein metrics. Actually what we really require is $a_{i,p}(M_1,g_1)=a_{i,p}(M_2,g_2)~(i=0,1,2)$, which is weaker than the isospectral condition $\text{Spec}^p(M_1,g_1)=\text{Spec}^p(M_2,g_2)$. Besides these, We also discuss an almost isospectral version proposed in \cite{WU}.

The rest of this paper is organized as follows. The main results and various consequences (Corollaries \ref{coro1} and \ref{coro2}) are stated in Section \ref{main results}. After presenting some preliminaries in Section \ref{preliminaries}, the next four sections are devoted to the proofs of these results respectively.

\section{Main results}\label{main results}
Let $(p_i,m_i)$ be determined by the recursive formula
\begin{eqnarray}\label{recursiveformula1}
\left\{ \begin{array}{ll}
({p}_1,{m}_1)=(0,1),\\

{p}_{i+1}={m}_i-{p}_i,\\

{m}_{i+1}=5{m}_i-6{p}_i+1.
\end{array} \right.
\end{eqnarray}
Here $(p_2,m_2)=(1,6)$, $(p_3,m_3)=(5,25)$, $(p_4,m_4)=(20,96)$, $\cdots$, whose distributions become more and more sparse as $i\rightarrow\infty$.

Our first simple result is a treatment of Question \ref{Q0} when $(M_1,g_1)$ is \emph{Einstein}, greatly extending \cite[Theorem 5.1]{Sa} and \cite[Theorem 3.3]{TK} for $p=0$ and $2$.
\begin{proposition}\label{main result-proposition}
Let $(M_1,g_1)$ and $(M_2,g_2)$ be $m$-dimensional Riemannian manifolds such that
$$a_{i,p}(M_1,g_1)=a_{i,p}(M_2,g_2),\qquad i=0,1,2$$
 for some \emph{fixed} $p$ with $p\leq [\frac{m}{2}]$, $(M_1,g_1)$ Einstein, and $(M_2,g_2)$ of constant sectional curvature $k$. If $(p,m)\notin\{(1,15),(2,15),(2,16),(p_i,m_i)~|~i\geq2\}$, where $(p_i,m_i)$ are determined by (\ref{recursiveformula1}), then $(M_1,g_1)$ is of constant sectional curvature $k$.
\end{proposition}

In recent years there have been much interest and increasing
research activities in Ricci solitons (\cite{Ca}), which are natural generalizations of Einstein metrics. The \emph{gradient Ricci soliton} on a Riemannian manifold $(M,g)$ is defined to be
\be\label{Ricci soliton equation}\text{Ric}(g)+\nabla^2f=\rho g,\qquad\rho\in\mathbb{R},\ee
where $\text{Ric}(g)$ is the Ricci tensor of $g$ and $f$ a smooth function on $M$ called a potential function of the Ricci soliton. The metric $g$ is Einstein if and only if $f$ is constant. When $M$ is compact, as we always assume in this article, and $\rho\leq 0$, it is well-known that the potential function $f$ in (\ref{Ricci soliton equation}) is necessarily constant (\cite[p. 123]{Ca}). So the interesting case is $\rho>0$ and called \emph{the gradient shrinking Ricci soliton}, which plays an important role in the
singularity study of the Ricci flow.

Our \emph{first major result} in this article is the following
\begin{theorem}\label{theorem1}
Let $(M_1,g_1)$ and $(M_2,g_2)$ be $m$-dimensional Riemannian manifolds such that
$$a_{i,p}(M_1,g_1)=a_{i,p}(M_2,g_2),\qquad i=0,1,2$$
for some \emph{fixed} $p$ with $p\leq [\frac{m}{2}]$, $(M_1,g_1,f)$ a gradient shrinking Ricci soliton, and $(M_2,g_2)$ of constant sectional curvature $k$.
Then $(M_1,g_1)$ is of
constant sectional curvature $k$ if $(p,m)$ falls into the following cases:
\begin{enumerate}
\item
\rm($p=0$, all dimensions $m$);

\item
($p=1$, all dimensions $m\geq 16$);

\item
($p=2,3,4,6,7,8$,
all dimensions $m\geq 2p+1$);

\item
($p=5$, all dimensions $m\geq 11$ and $m\neq25$);

\item
($p\geq 9$ and $p\not\in\{p_i~|~i\geq 4\}$, all dimensions $m$);

\item
($p=p_i$, all dimensions $m$ with at most one exception $m=m_i$)~\rm($i\geq 4$\rm),
\end{enumerate}
where $(p_i,m_i)$ are determined by (\ref{recursiveformula1}).
\end{theorem}

Recall that a compact K\"{a}hler manifold is called \emph{cohomologically Einstein} if its first Chern class and K\"{a}hler class are proportional, which particularly holds true if the second Betti number $b_2=1$. Let $(p_i,n_i)$ be determined by the following recursive formula
\begin{eqnarray}\label{recursiveformula}
\left\{ \begin{array}{ll}
({p}_1,{n}_1)=(1,3),\\

{p}_{i+1}=8{n}_i-5{p}_i+1,\\

{n}_{i+1}=19{n}_i-12{p}_i+3,
\end{array} \right.
\end{eqnarray}
where
$$(p_2,n_2)=(20,48), (p_3,n_3)=(285,675), (p_4,n_4)=(3976,9408), (p_5, n_5)=(55385,131043), \cdots,$$ and their distributions become more and more sparse as $i\rightarrow\infty$.

Our \emph{second major result} in this article is the following
\begin{theorem}\label{theorem2}
Let $(M_1,g_1,J_1)$ and $(M_2,g_2,J_2)$ be two complex $n$-dimensional K\"{a}hler manifolds such that
$$a_{i,p}(M_1,g_1)=a_{i,p}(M_2,g_2),\qquad i=0,1,2,$$
for a fixed $p$ with $p\leq n$, $(M_1,g_1,J_1)$ is cohomologically Einstein, and $(M_2,g_2,J_2)$ is of constant HSC $c$. Then $(M_1,g_1,J_1)$ is of constant HSC $c$ if $(p,n)$ falls into the following cases:
\begin{enumerate}
\item
\rm($p=0$, all dimensions $n$);

\item
($p=1$, all dimensions $n\geq 6$);

\item
($p=2$,
all dimensions $n$ with at most one exception $n=8$);

\item
($p\geq 3$ and $p\not\in\{p_i~|~i\geq 2\}$, all dimensions $n$);

\item
($p=p_i$, all dimensions $n$ with at most one exception $n=n_i$)~\rm($i\geq 2$\rm),
\end{enumerate}
where $(p_i,n_i)$ are determined by (\ref{recursiveformula}).
\end{theorem}
Consequently, Theorem \ref{theorem2} can be carried over to yield the same result when the HSC $c>0$, which amounts to $(M_2,g_2,J_2)=(\mathbb{C}P^n(c),g_0,J_0)$. In this situation we can, however, do one more case. Note that the exceptional case $(p=2,n=8)$ is not able to be dealt with in Theorem \ref{theorem2} due to the vanishing of a coefficient in the proof, which would be clear later. Nevertheless, thanks to a recent result due to Fujita (\cite{Fu}) solving a long-standing conjecture in complex geometry, the difficulty in this exceptional case for $(\mathbb{C}P^n(c),g_0,J_0)$ can be successfully overcome, which has been explained in \cite{Li} and shall be briefly reviewed again at the end of Section \ref{proof3}. In summary, we have the following partial affirmative answer towards Question \ref{Q2}.
\begin{corollary}\label{coro1}
Suppose that $(M,g,J)$ is a complex $n$-dimensional compact cohomologically Einstein K\"{a}hler manifolds such that $$a_{i,p}(M_1,g_1)=a_{i,p}(\mathbb{C}P^n(c),g_0),\qquad i=0,1,2,$$
for a fixed $p$ with $p\leq n$. Then $(M,g,J)$ is holomorphically isometric to $(\mathbb{C}P^n(c),g_0,J_0)$ if the pair $(p,n)$ satisfies one of the following cases:
\begin{enumerate}
\item
($p=0$, all dimensions $n$);

\item
($p=1$, all dimensions $n\geq 6$);

\item
($p=2$,
all dimensions $n$);

\item
($p\geq 3$ and $p\not\in\{p_i~|~i\geq 2\}$, all dimensions $n$);

\item
($p=p_i$, all dimensions $n$ with at most one exception $n=n_i$)~\rm($i\geq 2$\rm),
\end{enumerate}
where $(p_i,n_i)$ \rm($i\geq 2$\rm) are determined by (\ref{recursiveformula}).
\end{corollary}
When $p$ is \emph{even} and \emph{positive} and $\text{Spec}^p(M,g)=\text{Spec}^p(\mathbb{C}P^n(c),g_0)$, it turns out in \cite[Lemma 4.3]{Li} that the second Betti number $b_2(M)=1$ due to the Hard Lefschetz theorem and hence the condition of $(M,g,J)$ being cohomologically Einstein is automatically satisfied.
Also note that the positive integers $p_i$ determined by the recursive formula (\ref{recursiveformula1}) are even if and only if $i$ are even. Hence we have the following affirmative answer to Question \ref{Q2} for the following $(p,n)$ \emph{without any extra condition}, which is precisely the main result in \cite{Li}.
\begin{corollary}\label{coro2}
Assume that $p$ is even, positive and $p\leq n$. Then
\begin{enumerate}
\item
for $p=2$, Question \ref{Q2}
holds true in all dimensions $n$;

\item
for $p\geq 4$ and $p\not\in\{p_{2i}~|~i\geq 1\}$, Question \ref{Q2} holds true in all dimensions $n$;

\item
for $p=p_{2i}$, Question \ref{Q2} holds true in all dimensions $n$ with at most one exception $n=n_{2i}$.~\rm($i\geq 1$\rm),
\end{enumerate}
where $(p_{2i},n_{2i})$ are determined by (\ref{recursiveformula}).
\end{corollary}

\begin{remark}\label{remarks}

Theorem \ref{theorem2} and Corollaries \ref{coro1} and \ref{coro2} were obtained by the first author in a preprint \cite{Li2} posted on the arXiv in April 2018. For some reason it was not intended to be published. After collaborating with the second and third authors, we decide to put these results here. We also remark that Corollary \ref{coro1} was also reproved in \cite[Theorem 3.1]{HLXZ} by a somewhat different approach.
\end{remark}

Note that in Proposition \ref{main result-proposition}, Theorems \ref{theorem1}, \ref{theorem2}and Corollary \ref{coro1} above what we actually require are equalities of the first three $a_{i,p}$ in the asymptotic expansion (\ref{mpgformula}), to which the isospectral condition $\text{Spec}^p(M_1,g_1)=\text{Spec}^p(M_2,g_2)$ is only  \emph{sufficient} to lead. Next we shall discuss a weaker almost-isospectral condition. The following notion was introduced in \cite{WU}.

\begin{definition}
For some fixed $p$, the two $p$-spectral sets  $\{\lambda_{i,p}^{(1)}~|~i\geq1\}$ and $\{\lambda_{i,p}^{(2)}~|~i\geq1\}$ of Riemannian manifolds $(M_1,g_1)$ and $(M_2,g_2)$ are called \emph{$\alpha$-isospectral} ($\alpha\in\mathbb{R}$) if
\be\label{spectral}
\limsup_{i\rightarrow \infty}\frac{\big|\lambda_{i,p}^{(1)}-\lambda_{i,p}^{(2)}\big|}
{i^{-\alpha}}<\infty.
\ee
\end{definition}
The following fact shows that the first few coefficients $a_{i,p}$ of two $\alpha$-isospectral Riemannian manifolds are equal.
\begin{proposition}\label{almost isospectral lemma}
If two $m$-dimensional Riemannian manifolds $(M_1,g_1)$ and $(M_2,g_2)$ are $\alpha$-isospectral for some $p$, then
$$a_{i,p}(M_1,g_1)=a_{i,p}(M_2,g_2),\qquad i<1+\frac{m}{2}\min\{\alpha,1\}.$$
\end{proposition}
\begin{remark}\label{remark iso}
This fact is meaningful only when $\alpha>-\frac{2}{m}$.
The $p=0$ case of Proposition \ref{almost isospectral lemma} was announced in \cite[\S 3]{WU}. However, we have difficulties in understanding his proof. So we provide a proof ourselves in Section \ref{proof4}.
\end{remark}

By Proposition \ref{almost isospectral lemma}, when $m\geq3$ and $\alpha>\frac{2}{m}$, the first three coefficients $a_{0,p}$, $a_{1,p}$ and $a_{2,p}$ are equal. Hence we have
\begin{theorem}\label{almost isospectral theorem}
When $m\geq3$ (resp. $n\geq2$) and $\alpha>\frac{2}{m}$ (resp. $\alpha>\frac{1}{n}$), the conclusions in Proposition \ref{main result-proposition} and Theorem \ref{theorem1} (resp. Theorem \ref{theorem2} and Corollary \ref{coro1}) remain true if the condition ``$a_{i,p}(M_1,g_1)=a_{i,p}(M_2,g_2)~(i=0,1,2)$" is replaced by that of $\alpha$-isospectra.
\end{theorem}

\section{Preliminaries}\label{preliminaries}
In this section some necessary preliminaries are set up for our later purpose.

\subsection{Curvature decomposition}\label{subsection3.1}
Denote by $\mathring{\text{Ric}}(g)$ the traceless Ricci tensor  $$\mathring{\text{Ric}}(g):=\text{Ric}(g)-\frac{s_g}{m}g,$$ where $s_g$ is the scalar curvature function of $g$. Hence $g$ is Einstein if and only if $\mathring{\text{Ric}}(g)\equiv0$.

Recall that the $(0,4)$-type curvature tensor $R$ of an $m$-dimensional Riemannian manifold $(M,g)$ splits naturally into
three irreducible components under the orthogonal group (\cite[p. 47-48]{Be}):
\be\label{decompostion1}
\begin{split}
R&=\frac{s_g}{2m(m-1)}g\circ g+\frac{1}{m-2}\mathring{\text{Ric}}(g)\circ g+W\\
&=:S+P+W,
\end{split}\ee
where $W$ is called the Weyl curvature tensor and $``\circ"$ is the Kulkarni-Nomizu product
$$
A\circ B(x,y,z,w)=A(x,z)B(y,w)-A(x,w)B(y,z)+A(y,w)B(x,z)-A(y,z)B(x,w).
$$
Thus $S$ and $P$ involve the scalar curvature part and the traceless Ricci tensor part respectively, and $W$ exists as a nontrivial summand only when $m\geq 4$. The following facts are well-known.
\begin{lemma}\label{normrelation1}
The (pointwise) squared norms of $S$, $P$, $\text{\rm Ric}(g)$ and $\mathring{\text{\rm Ric}}(g)$ are related by
\begin{eqnarray}\label{normrelation11}
\left\{ \begin{array}{ll} |S|^2=\frac{2s_g^2}{m(m-1)},&(m\geq 2)\\
~\\
|P|^2=\frac{4}{m-2}|\mathring{\text{\rm Ric}}(g)|^2,&(m\geq 3)\\
~\\
|{\rm Ric}(g)|^2=|\mathring{\text{\rm Ric}}(g)|^2+\frac{s_g^2}{m}.\\
\end{array}\right.
\end{eqnarray}
The metric $g$ is of constant sectional curvature (resp. Einstein) if and only if  $P=W=0$ (resp. $P=0$). Moreover,
\be\label{normrelation33}|R|^2=|S|^2+|P|^2+|W|^2\ee
as the decomposition {\rm(\ref{decompostion1})} is orthogonal with respect to the norms.
\end{lemma}

$(M,g)$ is now further assumed to be a K\"{a}hler manifold with complex dimension $n$ (thus $m=2n$). The K\"{a}hler curvature tensor $R^c$, which is the complexification of the Riemannian curvature tensor, also splits into three irreducible components under unitary group bearing some
resemblance to (\ref{decompostion1}) (\cite[p. 77]{Be}):
\be\label{decompostion2}R^{c}=S^{c}+P^{c}+B,\ee
where $S^{c}$, $P^{c}$ and $B$ involve respectively the scalar curvature part, the traceless Ricci tensor part and what has now become known as the Bochner curvature tensor.
The reader may refer to \cite[Lemma 3.3]{Li0} and \cite[Lemma 3.3, Prop.3.4]{Li} for (\ref{normrelation22}) in the following facts.
\begin{lemma}\label{normrelation2}
If $(M,g)$ is a complex $n$-dimensional K\"{a}hler manifold, then the (pointwise) squared norms of $R$, $R^c$, $S^c$, $P^c$ and $\mathring{\text{\rm Ric}}(g)$ are related by
\be\label{normrelation22}
|R|^2=4|R^c|^2,\qquad
|S^c|^2=\frac{s^2_g}{2n(n+1)},\qquad
|P^c|^2=\frac{2}{n+2}|\mathring{\text{\rm Ric}}(g)|^2.\ee
The K\"{a}hler metric $g$ is of constant holomorphic sectional curvature  (resp. Einstein) if and only if $S^c=P^c\equiv0$ (resp. $P^c\equiv0$). Moreover
\be\label{normrelation44}|R^c|^2=|S^c|^2+|P^c|^2+|B|^2\ee
as the decomposition {\rm(\ref{decompostion2})} is orthogonal with respect to the norms.
\end{lemma}

\subsection{Patodi's formulas and applications}\label{subsection3.2}
As mentioned in the Introduction, for general $p$, the following explicit expressions for $a_{0,p}$, $a_{1,p}$ and $a_{2,p}$ in (\ref{mpgformula}) were determined by Patodi (\cite[Prop. 2.1]{Pa}) and have become indispensable tools in the study of related problems in spectral geometry.
\begin{lemma}[Patodi]
The coefficients $a_{0,p}$, $a_{1,p}$ and $a_{2,p}$ in the asymptotic formula (\ref{mpgformula}) are given by
\begin{eqnarray}\label{patodiformula}
\left\{ \begin{array}{ll} a_{0,p}={m\choose p}\text{\rm Vol}(M,g),\\
~\\
a_{1,p}=\frac{(m-2)!}{p!(m-p)!}\big[p^2-mp+\frac{m(m-1)}{6}\big]\displaystyle\int_{M}s_g\text{\rm dvol},\\
~\\
a_{2,p}=\displaystyle\int_{M}\big(c_1|\text{R}|^2+c_2|
\text{\rm Ric}(g)|^2+c_3s_g^2\big)\text{\rm dvol},
\end{array}\right.
\end{eqnarray}
where
\begin{eqnarray}\label{patodicoefficient}
\left\{ \begin{array}{ll} c_1=c_1(p,m)=\frac{1}{180}{m\choose p}
-\frac{1}{12}{m-2\choose p-1}+\frac{1}{2}{m-4\choose p-2},\\
~\\
c_2=c_2(p,m)=-\frac{1}{180}{m\choose p}
+\frac{1}{2}{m-2\choose p-1}-2{m-4\choose p-2},\\
~\\
c_3=c_3(p,m)=\frac{1}{72}{m\choose p}
-\frac{1}{6}{m-2\choose p-1}+\frac{1}{2}{m-4\choose p-2}.
\end{array}\right.
\end{eqnarray}
\end{lemma}

For our later use, we first apply various identities in Subsection \ref{subsection3.1} to reformulate $a_{2,p}$ for our later use.
\begin{lemma}
We have
\be\label{a2p real}
a_{2,p}=\int_{M}\Big\{
\big[\frac{2c_1}{m(m-1)}+\frac{c_2}{m}+c_3\big]s_g^{2}+
(\frac{4c_1}{m-2}+c_2)|\mathring{\text{\rm Ric}}(g)|^{2}+c_1|W|^2\Big\}\text{\rm dvol}. \ee
If $(M,g)$ is a complex $n$-dimensional K\"{a}hler manifold, then
\be\label{a2p complex}
a_{2,p}
=\int_{M}
\Big\{\big[\frac{2c_1}{n(n+1)}+
\frac{c_2}{2n}+c_3\big]s_{g}^2+\big(\frac{8c_1}{n+2}
+c_2\big)|\mathring{\text{\rm Ric}}(g)|^2+
4c_1|B|^2\Big\}{\rm dvol}.\ee
\end{lemma}
\begin{proof}
Applying (\ref{normrelation11}) and (\ref{normrelation33}) to the formula $a_{2,p}$ in (\ref{patodiformula}) leads to (\ref{a2p real}). For (\ref{a2p complex}), we apply (\ref{normrelation22}) and (\ref{normrelation44}) as well as (\ref{normrelation11}).
\end{proof}

In order to utilize the formula for $a_{1,p}$ in (\ref{patodiformula}), we need to require the coefficient in front of it be nonzero. The following lemma tells us that when it exactly happens and also explains why some sporadic cases $(p_i,m_i)$ \big(resp. $(p_i,n_i)$\big) determined by (\ref{recursiveformula1}) \big(resp. (\ref{recursiveformula})\big) in Proposition \ref{main result-proposition} and Theorem \ref{theorem1} (resp. Theorem \ref{theorem2}) have to be excluded.
\begin{lemma}\label{Pell equation}
The solutions $(p,m)\in\mathbb{Z}_{\geq0}\times\mathbb{Z}_{>0}$ of the equation
\be\label{coefficient equation}p^2-mp+\frac{m(m-1)}{6}=0\ee
satisfying $p\leq[\frac{m}{2}]$ are precisely $\big\{(p_i,m_i)~|~i\in\mathbb{Z}_{\geq1}\big\}$, where $(p_i,m_i)$ are determined by the recursive formula
\begin{eqnarray}\label{p m relation}
\left\{ \begin{array}{ll}
({p}_1,{m}_1)=(0,1),\\

{p}_{i+1}={m}_i-{p}_i,\\

{m}_{i+1}=5{m}_i-6{p}_i+1.
\end{array} \right.
\end{eqnarray}
Consequently, the solutions $(q,n)\in\mathbb{Z}_{\geq0}\times\mathbb{Z}_{>0}$  of the equation
\be\label{coefficient equation2}q^2-2nq+\frac{n(2n-1)}{3}=0\ee
satisfying $q\leq n$ are precisely $\big\{(q_i,n_i)~|~i\in\mathbb{Z}_{\geq1}\big\}$, where $(q_i,n_i)$ are determined by the recursive formula
\begin{eqnarray}\label{p n relation}
\left\{ \begin{array}{ll}
({q}_1,{n}_1)=(1,3),\\

{q}_{i+1}=8{n}_i-5{q}_i+1,\\

{n}_{i+1}=19{n}_i-12{q}_i+3.
\end{array} \right.
\end{eqnarray}
\end{lemma}
\begin{remark}
The recursive formula (\ref{p n relation}) for the equation (\ref{coefficient equation2}) has been established in \cite[\S 5.2]{Li}.
\end{remark}
\begin{proof}
By (\ref{coefficient equation}) we know
\be\label{equ}p=\frac{m\pm\sqrt{\frac{m(m+2)}{3}}}{2}.\ee
This implies that $m(m+2)=3r^2$ for some positive integer $r$, which is equivalent to
\be\label{pell2}(m+1)^2-3r^2=1.\ee
(\ref{pell2}) is the well-known \emph{Pell's equation}. If the (positive integer) solutions are ordered by magnitude, then the $i$-th solution $(r_i,m_i)~(i\in\mathbb{Z}_{\geq1})$ can be expressed in terms of the first one $(r_1,m_1)=(1,1)$ by (\cite[p. 11]{JW})
\be m_i+1+r_i\sqrt{3}=(m_1+1+r_1\sqrt{3})^i
=(2+\sqrt{3})^i.\nonumber\ee
Thus
\be\begin{split}
m_{i+1}+1+r_{i+1}\sqrt{3}&=(2+\sqrt{3})^{i+1}\\
&=(2+\sqrt{3})(m_i+1+r_i\sqrt{3})\\
&=(2m_i+3r_i+2)+(m_i+2r_i+1)\sqrt{3}.
\end{split}\nonumber\ee
This implies that the positive integer solutions $(r_i,m_i)$ to (\ref{pell2}) are given by the recursive formula
\begin{eqnarray}\label{r m relation}
\left\{ \begin{array}{ll}
({r}_1,{m}_1)=(1,1),\\

{r}_{i+1}={m}_i+2{r}_i+1,\\

{m}_{i+1}=2{m}_i+3{r}_i+1.
\end{array} \right.
\end{eqnarray}
The condition $p\leq[\frac{m}{2}]$ and (\ref{equ}) imply that $p_i=\frac{m_i-r_i}{2}$ and hence $r_i=m_i-2p_i$. Substituting it into (\ref{r m relation}) can easily lead to (\ref{p m relation}).

The recursive formula ${m}_{i+1}=5{m}_i-6{p}_i+1$ in (\ref{p m relation}) implies that the parities of $m_{i+1}$ and $m_i$ are different and hence $m_i$ is even if only if $i$ is even as $m_1=1$ is odd. Hence $(q_i,n_i):=(p_{2i},\frac{m_{2i}}{2})$ are precisely those solutions to
$$q^2-2nq+\frac{n(2n-1)}{3}=0.$$
So some easy calculations from (\ref{p m relation}) lead to the recursive formula for $(p_{2(i+1)},m_{2(i+1)})$ in terms of $(p_{2i},m_{2i})$ and therefore lead to (\ref{p n relation}).
\end{proof}

\subsection{A linear combination of Patodi's coefficients}\label{elementary}
We shall see in later sections that the last step in the proof of each main result comes down to showing the positivity of some linear combination of the coefficients $c_i$ in (\ref{patodicoefficient}). To this end, we treat the following elementary but useful facts, which can also be found in \cite[\S 5]{Li}.

Consider a linear combination of ${m\choose p}$, ${m-2\choose p-1}$ and ${m-4\choose p-2}$:
\be\begin{split}
&\alpha{m\choose p}
+\beta{m-2\choose p-1}+\gamma{m-4\choose p-2}\qquad(\alpha,\beta,\gamma\in\mathbb{R})\\
=&\frac{(m-4)!}{p!(m-p)!}\Big[\alpha m(m-1)(m-2)(m-3)\\
&+
\beta(m-2)(m-3)p(m-p)+\gamma p(p-1)(m-p)(m-p-1)\Big]\\
=:&\frac{(m-4)!}{p!(m-p)!}f(p,m,\alpha,\beta,\gamma)\\
=:&\frac{(m-4)!}{p!(m-p)!}\Big[\alpha m(m-1)(m-2)(m-3)+
g(p,m,\beta,\gamma)\Big],
\end{split}\nonumber\ee
and note that $f$ and $g$ satisfy
\be\label{duality}f(p,m,\alpha,\beta,\gamma)=f(m-p,m,\alpha,\beta,\gamma),\qquad
g(p,m,\beta,\gamma)=g(m-p,m,\beta,\gamma).\ee

Testing the positivity of this linear combination is equivalent to that of $f(p,m,\alpha,\beta,\gamma)$, which is reduced to estimating the minimal value of $g(p,m,\beta,\gamma)$. By direct calculations we have
\begin{lemma}\label{elementary lemma}
Fix $m$, $\beta$, and $\gamma\neq0$, and view $g(p,m,\beta,\gamma)=:g(p)$ as a quadratic polynomial of \emph{real} variable $p$. Then
\be\label{first}
g'(p)=(2p-m)\big[2\gamma p^2-2\gamma mp+\gamma(m-1)-\beta(m-2)(m-3)\big]
\ee
and hence $g'(p)$ has three roots
$p_1=\frac{m}{2}$, $p_2$ and $p_3$, where $p_{2,3}$ may be complex and satisfy $p_2+p_3=m$. 
Moreover,
\begin{eqnarray}\label{criticalpointvalue}
\left\{ \begin{array}{ll}
g''(p_1)=
(-2\beta-\gamma)m^2+(10\beta+2\gamma)m+(-12\beta-2\gamma)\\
~\\
g''(p_{2,3})=2\gamma(2p_{2,3}-m)^2.
\end{array} \right.
\end{eqnarray}
The Discriminant $\Delta$ arising from (\ref{first}) is defined to be
\be\Delta:=4\gamma^2(m^2-2m+2)+8\gamma\beta(m-2)(m-3).\ee
If $\Delta>0$, then $p_{2,3}$ are real and different from $\frac{m}{2}$ and hence the sign of $g''(p_{2,3})$ depends on that of $\gamma$.
\end{lemma}

\section{Proof of Proposition \ref{main result-proposition}}\label{proof1}
The case $(p,m)=(p_1,m_1)=(0,1)$ can be treated directly, which has been done in \cite{Ta1}. In view of Lemma \ref{Pell equation}, Proposition \ref{main result-proposition} is equivalent to the following
\begin{proposition}\label{main result-proposition again}
Let $(M_1,g_1)$ and $(M_2,g_2)$ be $m$-dimensional Riemannian manifolds such that
\be\label{three equal}a_{i,p}(M_1,g_1)=a_{i,p}(M_2,g_2),\qquad i=0,1,2\ee
 for some \emph{fixed} $p$ with $p\leq [\frac{m}{2}]$, $(M_1,g_1)$ Einstein, and $(M_2,g_2)$ of constant sectional curvature $k$. If the pair $(p,m)$ satisfies
\be\label{noneuqal}p^2-mp+\frac{m(m-1)}{6}\neq0\ee
and $(p,m)\notin\{(1,15),(2,15),(2,16)\}$, then $(M_1,g_1)$ is of constant sectional curvature $k$.
\end{proposition}
\begin{proof}
Combining the conditions (\ref{three equal}) and (\ref{noneuqal}) with Patodi's formula (\ref{patodiformula}) yields
\be\label{three}
\text{Vol}(M_1,g_1)=\text{Vol}(M_2,g_2),\qquad
\int_{M_1}s_{g_1}\text{dvol}=\int_{M_2}s_{g_2}\text{dvol},\qquad
a_{2,p}(M_1,g_1)=a_{2,p}(M_2,g_2).\ee
Since $(M_1,g_1)$ is Einstein and $(M_2,g_2)$ is of constant sectional curvature $k$, $s_{g_1}$ is constant and $s_{g_2}=m(m-1)k$. This, together with (\ref{three}), yields that
\be\label{scalar equ}s_{g_1}=s_{g_2}=m(m-1)k.\ee
Meanwhile, (\ref{a2p real}) leads to
\be\label{equ1} a_{2,p}(M_1,g_1)=
\int_{M_1}\Big\{
\big[\frac{2c_1}{m(m-1)}+\frac{c_2}{m}+c_3\big]s_{g_1}^{2}+
c_1|W(g_1)|^2\Big\}\text{\rm dvol}
\ee
and
\be\label{equ2}a_{2,p}(M_2,g_2)=
\int_{M_2}\big[\frac{2c_1}{m(m-1)}+\frac{c_2}{m}+c_3\big]
s_{g_2}^{2}\text{\rm dvol}\ee
Putting (\ref{three}), (\ref{scalar equ}), (\ref{equ1}) and (\ref{equ2}) together yields $\int_{M_1}c_1|W(g_1)|^2\text{dvol}=0$. If $c_1\neq0$ under our consideration, then $W(g_1)\equiv0$ and hence $g_1$ is of constant sectional curvature and (\ref{scalar equ}) implies that the constant is precisely $k$. So our conclusion follows from the following lemma.
\end{proof}

\begin{lemma}\label{c1}
Suppose that $p\leq[\frac{m}{2}]$. The coefficient $c_1=c_1(p,m)=0$ in (\ref{patodicoefficient}) if and only if $(p,m)\in\{(1,15),(2,15),(2,16)\}$. Moreover, $c_1>0$ in such cases:
$$(p=0, \text{all $m$}),\qquad(p=1, m\geq16),\qquad (p=2, m\neq15, 16),\qquad(p\geq3,\text{all $m$}).$$
\end{lemma}
\begin{proof}
Clearly $c_1(0,m)=\frac{1}{180}>0$ and $c_1(1,m)=\frac{1}{180}m-\frac{1}{12}$, which is zero if and only if $m=15$.

Now we assume that $2\leq p\leq[\frac{m}{2}]$ and hence $m\geq4$. With the notation in Subsection \ref{elementary} in mind, it suffices to treat $f(p,m,\frac{1}{180},
-\frac{1}{12},\frac{1}{2})$. Applying Lemma \ref{elementary lemma} in this case yields the desired Lemma \ref{c1} and details can be found in \cite[p. 1141]{Li}.

\end{proof}
\begin{remark}

In \cite[Theorem 3.3]{TK}, Gr.Tsagas and C.Kockinos proved Proposition \ref{main result-proposition again} (=Proposition \ref{main result-proposition}) for $p=2$ and $m\in [2,7],m=14$ or $m\geq 17$. In fact, we do not need the positivity of a term called $Q_{3}$ in \cite{TK}. So for $(p=2, 8\leq m\leq 13)$ it still holds.
\end{remark}

\section{Proof of Theorem \ref{theorem1}}\label{proof2}
Let $(M,g,f)$ be a gradient shrinking Ricci soliton on an $m$-dimensional Riemannian manifold. Then we have
\be\label{Ricci soliton eq}\text{Ric}_{ij}+\nabla_i\nabla_j f=\rho g_{ij},\qquad\rho\in\mathbb{R}_{>0}.\ee

We shall in this situation express $a_{2,p}$ in terms of the potential function $f$. To this end, the following identity is needed.
\begin{lemma}
The gradient shrinking Ricci soliton $(M,g,f)$ satisfies
\be\label{integral identity}
\int_{M}\lvert \nabla^{2}f\rvert^{2}\text{\rm dvol}=\frac{1}{2}\int_{M}(\Delta f)^{2}\text{\rm dvol}.
\ee
\end{lemma}
\begin{proof}
By taking the traces on both sides of (\ref{Ricci soliton eq}) we have
\be\label{one}s_g+\Delta f=\rho m.\ee
As showed in \cite[p. 123]{Ca}, the following formula can be read off from (\ref{Ricci soliton eq}).
\begin{equation}\label{nabal R}
\nabla_{i}s_g=2\text{Ric}_{ij}\nabla_jf.
\end{equation}
By Bochner's formula, we have
\begin{equation}\label{Bochner}
\frac{1}{2}\Delta \lvert\nabla f\rvert^{2}=\langle \nabla \Delta f, \nabla f\rangle+ \lvert\nabla^{2}f\rvert^{2}+\text{Ric}(\nabla f,\nabla f).
\end{equation}
Integrating the equation (\ref{Bochner}) yields
\be\begin{split}
\int_{M}\lvert\nabla^{2}f\rvert^{2}\text{dvol}&=-\int_{M}\langle \nabla \Delta f,\nabla f\rangle\text{dvol}-\int_{M}\text{Ric}(\nabla f,\nabla f)\text{dvol}\\
&=\int_{M}(\Delta f)^{2}\text{dvol}-\frac{1}{2}\int_{M}\langle \nabla s_g, \nabla f \rangle\text{dvol}\qquad\big(\text{by (\ref{nabal R})}\big) \\
&=\int_{M}(\Delta f)^{2}\text{dvol}+\frac{1}{2}\int_{M}s_g\Delta f\text{dvol}\\
&=\int_{M}(\Delta f)^{2}\text{dvol}+\frac{1}{2}\int_{M}(m\rho-\Delta f) \Delta f\text{dvol}\qquad\big(\text{by (\ref{one})}\big)\\
&=\frac{1}{2}\int_{M}(\Delta f)^{2}\text{dvol}.
\end{split}\ee
\end{proof}
With (\ref{integral identity}) in hand, the following expression for $a_{2,p}$ can be derived.
\begin{lemma}
For the gradient shrinking Ricci soliton $(M,g,f)$, we have
\be\label{a2p soliton}
a_{2,p}=\big[\frac{2c_1}{m(m-1)}+\frac{c_2}{m}+c_3\big]
\frac{(\int_Ms_g\text{\rm dvol})^2}{\text{\rm Vol}(M)}+\int_M\big[(\frac{2c_1}{m-1}+\frac{c_2}{2}+c_3)
(\Delta f)^2+c_1|W|^2\big]\text{\rm dvol}.\ee
\end{lemma}
\begin{proof}
\be\label{S express}\begin{split}
\int_{M}s_g^{2}\text{dvol}&=\int_{M}(\rho m-\Delta f)^{2}\text{dvol}\qquad\big(\text{by (\ref{one})}\big)\\
&=\int_{M}\big[(\rho m)^2-2\rho m\Delta f+(\Delta f)^{2}\big]\text{dvol}\\
&=(\rho m)^2\text{Vol}(M)+\int_{M}(\Delta f)^{2}\text{dvol}\\
&=\frac{(\int_Ms_g\text{\rm dvol})^2}{\text{\rm Vol}(M)}+
\int_{M}(\Delta f)^{2}\text{dvol}.\qquad\big(\text{by (\ref{one})}\big)
\end{split}\ee
\be\begin{split}\label{traceless Ricci}
\int_M|\mathring{\text{\rm Ric}}(g)|^2\text{dvol}&=\int_M\Big[|\text{\rm Ric}(g)|^2-\frac{s_g^2}{m}\Big]\text{dvol}\qquad\big(\text{by (\ref{normrelation11})}\big)\\
&=\int_M\Big[|\rho g_{ij}-\nabla_i\nabla_j f|^2-\frac{1}{m}(\rho m-\Delta f)^2\Big]\text{dvol}\qquad\big(\text{by (\ref{Ricci soliton eq}) and (\ref{one})}\big)\\
&=\int_M\Big[|\nabla^{2}f|^{2}-\frac{(\Delta f)^{2}}{m}\Big]\text{dvol}\\
&=\frac{m-2}{2m}\int_M(\Delta f)^{2}\text{dvol}\qquad\big(\text{by (\ref{integral identity})}\big)
\end{split}\ee
Substituting (\ref{S express}) and (\ref{traceless Ricci}) into (\ref{a2p real}) gives the desired (\ref{a2p soliton}).
\end{proof}
Now we are in a position to show
\begin{proposition}\label{prop}
Suppose $(M_1,g_1,f)$ is a gradient shrinking Ricci soliton and the Riemannian manifold $(M_2,g_2)$ is of constant sectional curvature $k$. If \be\label{three equal again}a_{i,p}(M_1,g_1)=a_{i,p}(M_2,g_2),\qquad i=0,1,2\ee
 for some \emph{fixed} $p$ with $p\leq [\frac{m}{2}]$, and
\be\label{nonequal again}
c_{1}> 0,\qquad \frac{2c_{1}}{m-1}+\frac{c_2}{2}+c_{3}>0,\qquad p^2-mp+\frac{m(m-1)}{6}\neq0,
\ee
then $(M_1,g_1)$ is also of constant sectional curvature $k$.
\end{proposition}
\begin{proof}
As before, (\ref{three equal again}), the third one in (\ref{nonequal again}) and Patodi's formula (\ref{patodiformula}) yield
\be\label{three again}
\text{Vol}(M_1,g_1)=\text{Vol}(M_2,g_2),\qquad
\int_{M_1}s_{g_1}\text{dvol}=\int_{M_2}s_{g_2}\text{dvol},\qquad
a_{2,p}(M_1,g_1)=a_{2,p}(M_2,g_2).\ee
By (\ref{a2p soliton}) and the condition that $(M_2,g_2)$ be of constant sectional curvature, we have
$$a_{2,p}(M_2,g_2)=\big[\frac{2c_1}{m(m-1)}+\frac{c_2}{m}+c_3\big]
\frac{(\int_{M_2}s_{g_2}\text{\rm dvol})^2}{\text{\rm Vol}(M_2)}$$
and hence
\be\label{equaility}\int_{M_1}\big[(\frac{2c_1}{m-1}+\frac{c_2}{2}+c_3)
(\Delta f)^2+c_1|W(g_1)|^2\big]\text{\rm dvol}=0\ee
due to (\ref{a2p soliton}), $a_{2,p}(M_1,g_1)=a_{2,p}(M_2,g_2)$
and the first two identities in (\ref{three again}). The positivity of the two coefficients in (\ref{equaility}) as assumed in (\ref{nonequal again}) implies that $f$ be constant (hence $g_1$ be Einstein) and $W(g_1)\equiv0$. So $(M_1,g_1)$ is of constant sectional curvature, whose constant is also $k$ as $s_{g_1}=s_{g_2}$.
\end{proof}
With Proposition \ref{prop}, Lemmas \ref{Pell equation} and \ref{c1} in hand, Theorem \ref{theorem1} now follows from the following
\begin{lemma}\label{coefficient lemma}
Suppose $p\leq[\frac{m}{2}]$. Then $\frac{2c_{1}}{m-1}+\frac{c_2}{2}+c_{3}>0$ in such cases: $$(p=0,m\geq1),\qquad (p=1,m\geq3),\qquad (p\geq2,m\geq17),\qquad(p<\frac{m}{2}, 5\leq m\leq 16).$$
\end{lemma}
\begin{proof}
We only treat the cases $(p\geq2,m\geq17)$ and $(p<\frac{m}{2}, 5\leq m\leq 16)$  as the first two are quite direct to check. Note that
$$\frac{2c_{1}}{m-1}+\frac{c_2}{2}+c_{3}=
\frac{m}{90(m-1)}{m\choose p}+\frac{m-3}{12(m-1)}{m-2\choose p-1}-\frac{m-3}{2(m-1)}{m-4\choose p-2}.$$
With the notation in Subsection \ref{elementary} in hand, we need to show \be\label{six}f\big(p,m,\frac{m}{90(m-1)},\frac{m-3}{12(m-1)},
-\frac{m-3}{2(m-1)}\big)>0.\ee

Applying Lemma \ref{elementary lemma} we have
$$g''(p_1=\frac{m}{2})=\frac{m(m-3)(2m-1)}{6(m-1)}>0,\qquad g''(p_{2,3})<0$$
as the discriminant
$$\Delta=\frac{m(m-3)^2(2m-1)}{3(m-1)^2}>0.$$
This implies that $p_1=\frac{m}{2}$ is the unique local minimal-value point of $g(p)$ in the interval $(2,m-2)$. Thus
\be\label{four}\min_{p\in[2,m-2]}f\big(p,m,\frac{m}{90(m-1)},\frac{m-3}{12(m-1)},
-\frac{m-3}{2(m-1)}\big)=\min\{f\big|_{p=2},f\big|_{p=\frac{m}{2}}\}.\ee

On the other hand, direct calculations illustrate that
\be\label{five}f\big|_{p=2}=\frac{(m-2)(m-3)}{90(m-1)}(m^3+14m^2-165m+360),\qquad
f\big|_{p=\frac{m}{2}}=\frac{m^2(m-2)(m-3)(m-16)}{1440(m-1)}.\ee
Combining (\ref{four}) with (\ref{five}) yields (\ref{six}) when $(p\geq2, m\geq17)$.

Then we treat the case $(p<\frac{m}{2}, 5\leq m\leq 16)$. 
From above, there is no local interior minimum point of $g(p)$ for $0\leq p\leq\frac{m-1}{2}$. Hence 
\be\label{seven}\min_{p\in[0,\frac{m-1}{2}]}f\big(p,m,\frac{m}{90(m-1)},\frac{m-3}{12(m-1)},
-\frac{m-3}{2(m-1)}\big)=\min\{f\big|_{p=0},f\big|_{p=\frac{m-1}{2}}\}.\ee
Direct calculations give that when $5\leq m\leq 16$
\be\label{eight}f\big|_{p=0}=\frac{(m-2)(m-3)m^{2}}{90}>0,\qquad
f\big|_{p=\frac{m-1}{2}}=\frac{(m-3)[m(m-9)^{2}+(m-3)^{2}+36]}{1440}>0.\ee
Combining (\ref{seven}) with (\ref{eight}) yields (\ref{six}) when $(p<\frac{m}{2}, 5\leq m\leq 16)$.
\end{proof}

\section{Proof of Theorem \ref{theorem2}}\label{proof3}
Let $(M,g,J)$ be a K\"{a}hler manifold with complex dimension $n\geq2$, i.e., $J$ is an integrable complex structure and $g$ a $J$-invariant Riemannian metric, and define
$\omega:=\frac{1}{2\pi}g(J\cdot,\cdot)$, the K\"{a}hler form of $g$. In our notation of $\omega$,
\be\label{volumeelement}\text{the volume element of $(M,g)=\frac{\pi^n}{n!}\omega^n$}.\ee
With the notation understood, we have (cf. \cite[Lemma 3.5]{Li}).
\begin{lemma}\label{integralformulas}
Suppose that $(M,g,J)$ is a K\"{a}hler manifold with complex dimension $n\geq2$. Then
\be\label{integralformula1}
\int_Mc_1(M)\wedge[\omega]^{n-1}=
\frac{1}{2n}\int_M{s_g}\cdot\omega^n,
\ee
and
\be\label{integralformula2}
\int_Mc_1^2(M)\wedge[\omega]^{n-2}=
\int_M
\Big(\frac{n-1}{2n}s^2_g-|\mathring{\text{\rm Ric}}(g)|^2\Big)
\cdot\frac{\omega^n}{2n(n-1)}.\ee
\end{lemma}
\begin{remark}
A different notation $\tilde{R}ic(\omega)$ is used in \cite[Lemma 3.5]{Li}, and it is related to $\mathring{\text{Ric}}(g)$ by $|\mathring{\text{\rm Ric}}(g)|^2=2|\tilde{R}ic(\omega)|^2$.
\end{remark}

In the sequel of this section we \emph{assume} that the two complex $n$-dimensional compact K\"{a}hler manifolds $(M_1,g_1,J_1)$ and $(M_2,g_2,J_2)$ ($n\geq2$) satisfy the conditions in Theorem \ref{theorem2}. Namely, $(M_1,g_1,J_1)$ is cohomologically Einstein, i.e., $c_1(M_1)\in\mathbb{R}[\omega_1]$, $(M_2,g_2,J_2)$ is of constant HSC $c$ \big(hence $s_{g_2}=n(n+1)c$\big), and $a_{i,p}(M_1,g_1)=a_{i,p}(M_2,g_2)$ for $i=0,1,2$.
\begin{lemma}\label{keylemma}
Assume that $p^2-2np+\frac{n(2n-1)}{3}\neq0$. Then
\be\label{key2}\int_{M_1}|\mathring{\text{\rm Ric}}(g_1)|^2{\rm dvol}=
\frac{n-1}{2n}\int_{M_1}\big\{s_{g_1}^2-[n(n+1)c]^2\big\}{\rm dvol},\ee
and
\be\label{a2b}
\frac{2n}{n-1}\big[\frac{4n+2}{(n+1)(n+2)}c_1+
\frac{1}{2}c_2+c_3\big]
\int_{M_1}|\mathring{\text{\rm Ric}}(g_1)|^2
{\rm dvol}+4c_1\int_{M_1}|B(g_1)|^2{\rm dvol}=0.
\ee
\end{lemma}
\begin{proof}
Combining the conditions $a_{i,p}(M_1,g_1)=a_{i,p}(M_2,g_2)$ ($i=0,1$) with Patodi's formula (\ref{patodiformula}) implies that
\be\label{a01}\text{Vol}(M_1,g_1)=\text{Vol}(M_2,g_2),
\qquad\int_{M_1}s_{g_1}{\rm dvol}=\int_{M_2}n(n+1)c{\rm dvol}.\ee

Observe from (\ref{volumeelement}) that $\omega_i^n$ ($i=1,2$) are volume forms up to a universal constant. Therefore
\be\label{ric}
\begin{split}
\int_{M_1}
\big(\frac{n-1}{2n}s^2_{g_1}-
|\mathring{\text{\rm Ric}}(g_1)|^2\big)
\cdot\frac{\omega_1^n}{2n(n-1)}
=&
\big(\int_{M_1}c_1^2(M_1)\wedge[\omega_1]^{n-2}\big)
\qquad\big(\text{by (\ref{integralformula2})}\big)\\
=&\frac{\big(\int_{M_1}c_1(M_1)\wedge[\omega_1]^{n-1}\big)^2}
{\int_{M_1}\omega_1^n}
\qquad\Big(\text{by $c_1(M_1)\in\mathbb{R}[\omega_1]$}\Big)\\
=&\frac{\big(\int_{M_1}s_{g_1}\omega_1^{n}\big)^2}
{4n^2\int_{M_1}\omega_1^n}
\qquad\big(\text{by (\ref{integralformula1})}\big)\\
=&\frac{\big(\int_{M_2}n(n+1)c\omega_2^{n}\big)^2}
{4n^2\int_{M_2}\omega_2^n}
\qquad\big(\text{by (\ref{a01})}\big)\\
=&\frac{[n(n+1)c]^2}{4n^2}\int_{M_2}\omega_2^n\\
=&\frac{[n(n+1)c]^2}{4n^2}\int_{M_1}\omega_1^n.
\qquad\big(\text{by (\ref{a01})}\big)
\end{split}
\ee
Rewriting (\ref{ric}) by singling out the term $|\mathring{\text{\rm Ric}}(g_1)|^2$ leads to the desired equality (\ref{key2}).

For (\ref{a2b}), the formula (\ref{a2p complex}) for K\"{a}hler manifolds yields
\be\label{integralformula4}a_{2,p}(M_2,g_2)=\int_{M_2}
\big[\frac{2c_1}{n(n+1)}+
\frac{1}{2n}c_2+c_3\big][n(n+1)c]^2{\rm dvol}.\ee
The condition $a_{2,p}(M_1,g_1)=a_{2,p}(M_2,g_2)$ then implies that
\be\label{key3}\begin{split}
0=&\int_{M_1}\big[\frac{2}{n(n+1)}c_1+
\frac{1}{2n}c_2+c_3\big]\big\{s_{g_1}^2-[n(n+1)c]^2\}{\rm dvol}\\
&+\big(\frac{8}{n+2}c_1
+c_2\big)\int_{M_1}|\mathring{\text{\rm Ric}}(g_1)|^2{\rm dvol}+
4c_1\int_{M_1}|B(g_1)|^2{\rm dvol}\\
=&\int_{M_1}\frac{2n}{n-1}\big[\frac{2}{n(n+1)}c_1+
\frac{1}{2n}c_2+c_3\big]
|\mathring{\text{\rm Ric}}(g_1)|^2{\rm dvol}\\
&+\big(\frac{8}{n+2}c_1
+c_2\big)\int_{M_1}|\mathring{\text{\rm Ric}}(g_1)|^2{\rm dvol}+
4c_1\int_{M_1}|B(g_1)|^2{\rm dvol}\qquad\big(\text{by (\ref{key2})}\big)\\
=&\frac{2n}{n-1}\big[\frac{4n+2}{(n+1)(n+2)}c_1+
\frac{1}{2}c_2+c_3\big]
\int_{M_1}|\mathring{\text{\rm Ric}}(g_1)|^2
{\rm dvol}+4c_1\int_{M_1}|B(g_1)|^2{\rm dvol}.
\end{split}\ee
This yields the desired equality (\ref{a2b}).
\end{proof}

Lemma \ref{keylemma} yields the following consequence.
\begin{proposition}\label{prop2}
Let $(M_i,g_i,J_i)~(i=1,2)$ be as above and the pair $(p,n)$ satisfies the following restrictions
$$p^2-2np+\frac{n(2n-1)}{3}\neq0,\qquad \frac{4n+2}{(n+1)(n+2)}c_1+
\frac{1}{2}c_2+c_3>0,\qquad c_1>0.$$
Then $(M_1,g_1,J_1)$ is of constant HSC $c$.
\end{proposition}
\begin{proof}
Via (\ref{a2b}), these restrictions on $(p,n)$ imply that $\mathring{\text{\rm Ric}}(g_1)=0$ and $B(g_1)=0$. These tell us that the metric $g_1$ is of constant HSC, whose value is precisely $c$ due to (\ref{a01}).
\end{proof}

In view of Proposition \ref{prop2}, (the second part of) Lemma \ref{Pell equation} and Lemma \ref{c1}, Theorem \ref{theorem2} follows from the following
\begin{lemma}
Assume that $p\leq n$, $2n=m$ and $n\geq2$. Then
$$\frac{4n+2}{(n+1)(n+2)}c_1+
\frac{1}{2}c_2+c_3>0.$$
\end{lemma}
\begin{proof}
The cases $p=0$ and $p=1$ can be directly checked. For $p\geq2$, the strategy is similar to that in the proof of Lemma \ref{coefficient lemma} and has been carried out in \cite[Prop. 4.5, \S 5.1]{Li}.
\end{proof}

Let us end this section by explaining that how the exceptional case $(p,n)=(2,8)$ can be dealt with in the case of the HSC $c>0$, i.e., in the case of $(M_2,g_2,J_2)=\big(\mathbb{C}P^8(c),g_0,J_0\big)$, due to a recent result of Fujita (\cite{Fu}). If $(p,n)=(2,8)$, then $c_1(2,16)=0$ due to Lemma \ref{c1}, i.e., in (\ref{a2b}) the coefficient in front of $|B(g_1)|^2$  vanishes and from the proof of Proposition \ref{prop2} we are \emph{not} able to conclude that $B(g_1)=0$ but \emph{only} conclude that the K\"{a}hler metric $g_1$ is Einstein. Nevertheless, if the constant HSC $c$ in question is \emph{positive}, then in this case $(M_1,g_1,J_1)$ is a \emph{Fano K\"{a}hler-Einstein} manifold. An equivalent form of the main result in \cite{Fu} (see \cite[Theorem 2.2]{Li}) states that if both the scalar curvature and the volume of a Fano K\"{a}hler-Einstein manifold are equal to those of $\big(\mathbb{C}P^n(c),g_0,J_0\big)$, then it is holomorphically isometric to $\big(\mathbb{C}P^n(c),g_0,J_0\big)$, from which and (\ref{a01}) the exceptional case $(p,n)=(2,8)$ in Corollary \ref{coro1} follows.

\section{Proof of Proposition \ref{almost isospectral lemma}}\label{proof4}
Let $\{\lambda_{i,p}^{(1)}~|~i\geq1\}$ and $\{\lambda_{i,p}^{(2)}~|~i\geq1\}$ be the $p$-spectral sets of Riemannian manifolds $(M_1,g_1)$ and $(M_2,g_2)$ respectively.

\subsection{Preliminaries}
First we derive the following
\begin{lemma}\label{V1=V2}
The volumes $V_1$ and $V_2$ are equal and hence $a_{0,p}^{(1)}=a_{0,p}^{(2)}$, where $V_i:=\text{Vol}(M_i,g_i)$.
\end{lemma}
\begin{proof}
Weyl's asymptotic formula implies that
\be\label{Weyl}
\lim_{i\rightarrow\infty}\frac{\lambda_{i,p}^{(1)}}
{i^{\frac{2}{m}}}=C_{m,p}V_1^{-\frac{2}{m}},\qquad
\lim_{i\rightarrow\infty}\frac{\lambda_{i,p}^{(2)}}
{i^{\frac{2}{m}}}=C_{m,p}V_2^{-\frac{2}{m}},\ee
where $C_{m,p}$ is a universal positive constant depending on $m$ and $p$. Note that
\be\begin{split}
&\limsup_{i\rightarrow\infty}\frac{\big|\lambda_{i,p}^{(1)}-\lambda_{i,p}^{(2)}\big|}
{i^{-\alpha}}\\
=&\limsup_{i\rightarrow\infty}\Big|\big[\frac{\lambda_{i,p}^{(1)}}{i^{\frac{2}{m}}}-C_{m,p}V_1^{-\frac{2}{m}}\big]
-\big[\frac{\lambda_{i,p}^{(2)}}{i^{\frac{2}{m}}}-C_{m,p}V_2^{-\frac{2}{m}}\big]
+C_{m,p}\big(V_1^{-\frac{2}{m}}-V_2^{-\frac{2}{m}}\big)
\Big|i^{\frac{2}{m}+\alpha}.\end{split}
\ee
As commented in Remark \ref{remark iso}, we may assume that $\alpha>-\frac{2}{m}$. So combining the conditions (\ref{spectral}), (\ref{Weyl}) and $\frac{2}{m}+\alpha>0$ yields that $V_1=V_2$.
\end{proof}
(\ref{Weyl}) can now be rewritten as
\be\label{Weyl2}
\lim_{i\rightarrow\infty}\frac{\lambda_{i,p}^{(1)}}
{i^{\frac{2}{m}}}=\lim_{i\rightarrow\infty}\frac{\lambda_{i,p}^{(2)}}
{i^{\frac{2}{m}}}=C_{m,p}V^{-\frac{2}{m}},\qquad V:=V_1=V_2.\ee

By (\ref{mpgformula}) we have
\be\label{2}\sum_{i=0}^{\infty}\big[\exp(-\lambda_{i,p}^{(1)}t)-
\exp(-\lambda_{i,p}^{(2)}t)\big]=
\frac{1}{(4\pi)^{\frac{m}{2}}}\sum_{i=0}^N
\big[a_{i,p}^{(1)}-a_{i,p}^{(2)}\big]
t^{i-\frac{m}{2}}+O(t^{N-\frac{m}{2}+1}).\ee
On the other hand,
\be\label{3}\begin{split}
&\Big|\sum_{i=0}^{\infty}\big[\exp(-\lambda_{i,p}^{(1)}t)-
\exp(-\lambda_{i,p}^{(2)}t)\big]\Big|\\
\leq& t\sum_{i=1}^{\infty}\exp(-\lambda_{i,p}t)\big|\lambda_{i,p}^{(1)}-\lambda_{i,p}^{(2)}\big|
\qquad\Big(\lambda_{i,p}:=\min\{\lambda_{i,p}^{(1)},\lambda_{i,p}^{(2)}\}\Big)\\
\leq& Ct\sum_{i=1}^{\infty}\exp(-\lambda_{i,p}t)i^{-\alpha},
\qquad\big(\text{by (\ref{spectral})}\big)\\
\end{split}\ee
where $C$ is some positive constant.

We shall show in the next subsection that Proposition \ref{almost isospectral lemma} can be derived from the following lemma.
\begin{lemma}\label{lastlemma}
Let $\omega(t):=\sum_{i=1}^{\infty}\exp(-\lambda_{i,p}t)i^{-\alpha}$, which satisfies
\begin{eqnarray}\label{omegacondition}
\left\{ \begin{array}{ll}
\text{$\omega(t)$ is bounded above},&\alpha>1,\\
\limsup_{t\rightarrow0^+}
\frac{\omega(t)}{t^{\frac{m}{2}(\alpha-1)}}<\infty,&\alpha<1,\\
\limsup_{t\rightarrow0^+}\frac{\omega(t)}{\ln\frac{1}{t}}<\infty,&\alpha=1.\\
\end{array} \right.
\end{eqnarray}
\end{lemma}

\subsection{Proof of Proposition \ref{almost isospectral lemma}}
Assume for the moment the validity of (\ref{omegacondition}), we show in this subsection how to apply it to prove Proposition \ref{almost isospectral lemma}.

\begin{proof}
If $a_{i,p}^{(1)}=a_{i,p}^{(2)}$ for all $i$, Proposition \ref{almost isospectral lemma} clearly holds true. Hence by Lemma \ref{V1=V2} we may assume that $i_0\in\mathbb{Z}_{>0}$ is such that
\be\label{i_0}
\text{$a_{i_0,p}^{(1)}\neq a_{i_0,p}^{(2)}$ and
$a_{i,p}^{(1)}=a_{i,p}^{(2)}$ whenever $i<i_0$.}
\ee
Combining (\ref{2}) with (\ref{i_0}) we have
$$\limsup_{t\rightarrow0^+}
\frac{\Big|\sum_{i=0}^{\infty}\big[\exp(-\lambda_{i,p}^{(1)}t)-
\exp(-\lambda_{i,p}^{(2)}t)\big]\Big|}
{t^{i_0-\frac{m}{2}}}\in(0,\infty),$$
and hence due to (\ref{3}) we have
\be\label{8}\limsup_{t\rightarrow0^+}\frac{t\omega(t)}
{t^{i_0-\frac{m}{2}}}\in(0,\infty].\ee

\emph{Case $1$:} $\alpha>1$. In this case
\be\label{9}\frac{t\omega(t)}
{t^{i_0-\frac{m}{2}}}=\omega(t)\cdot t^{1+\frac{m}{2}-i_0}.\ee
Combining this with the first case in (\ref{omegacondition}) and (\ref{8}) yields that $1+\frac{m}{2}-i_0\leq0$ and therefore $i_0\geq1+\frac{m}{2}.$ By the definition (\ref{i_0}) of $i_0$, we know that $a_{i,p}^{(1)}=a_{i,p}^{(2)}$ whenever $i<1+\frac{m}{2}$ and prove Proposition \ref{almost isospectral lemma} in the case of $\alpha>1$.

\emph{Case $2$:} $\alpha<1$. In this case
\be\label{10}\frac{t\omega(t)}
{t^{i_0-\frac{m}{2}}}=\frac{\omega(t)}{t^{\frac{m}{2}(\alpha-1)}}\cdot t^{1+\frac{m}{2}\alpha-i_0}.\ee
Combining this with the second case in (\ref{omegacondition}) and (\ref{8}) yields $1+\frac{m}{2}\alpha-i_0\leq0$,, and hence proves it in this case due to the same reasoning as above.

\emph{Case $2$:} $\alpha=1$. In this case
\be\label{10}\frac{t\omega(t)}
{t^{i_0-\frac{m}{2}}}=\frac{\omega(t)}{\ln{\frac{1}{t}}}\cdot t^{1+\frac{m}{2}-i_0}\ln{\frac{1}{t}}.\ee
Combining this with the third case in (\ref{omegacondition}) and (\ref{8}) also yields $1+\frac{m}{2}-i_0\leq0$.
This, as above, proves Proposition \ref{almost isospectral lemma} in this case.
\end{proof}

\subsection{Proof of Lemma \ref{lastlemma}}
In order to complete the proof of Proposition \ref{almost isospectral lemma}, it suffices to show Lemma \ref{lastlemma}.
\begin{proof}
First
$$\omega(t)=\sum_{i=1}^{\infty}\exp(-\lambda_{i,p}t)i^{-\alpha}
<\sum_{i=1}^{\infty}i^{-\alpha}<\infty~\text{whenever $\alpha>1$}.$$
This proves the first case in (\ref{omegacondition}).

Next we treat the cases of $\alpha<1$ and $\alpha=1$.
Recall in (\ref{3}) that $\lambda_{i,p}=\min\{\lambda_{i,p}^{(1)},\lambda_{i,p}^{(2)}\}$
and so (\ref{Weyl2}) yields that
\be\label{4}\lim_{i\rightarrow\infty}\frac{\lambda_{i,p}}
{i^{\frac{2}{m}}}=C_{m,p}V^{-\frac{2}{m}}=:C_0.\ee
By (\ref{4}), for any $\epsilon>0$, there exists an $N\in\mathbb{Z}_{>0}$ such that
\be\label{5}\big|\frac{\lambda_{i,p}}
{i^{\frac{2}{m}}}-C_0\big|<\epsilon,\qquad \text{whenever $i>N$}.\ee

\emph{Case $2$:} $\alpha<1$. Then
\be\begin{split}
\omega(t)&=\sum_{i=1}^{\infty}e^{-\lambda_{i,p}
t}i^{-\alpha}\\
&=\sum_{i\leq N}e^{-\lambda_{i,p}
t}i^{-\alpha}+
\sum_{i> N}\exp\big[-\frac{\lambda_{i,p}}{i^{\frac{2}{m}}}
(it^{\frac{m}{2}})^{\frac{2}{m}}\big]i^{-\alpha}\\
&\leq\sum_{i\leq N}e^{-\lambda_{i,p}
t}i^{-\alpha}+
\sum_{i> N}\exp\big[-(C_0-\epsilon)(it^{\frac{m}{2}})^{\frac{2}{m}}\big]i^{-\alpha},
\qquad\big(\text{by (\ref{5})}\big)
\end{split}\ee
and hence
\be\label{6}
\frac{\omega(t)}{t^{\frac{m}{2}(\alpha-1)}}\leq
\sum_{i\leq N}e^{-\lambda_{i,p}
t}i^{-\alpha}t^{\frac{m}{2}(1-\alpha)}+
\sum_{i> N}\exp\big[-(C_0-\epsilon)(it^{\frac{m}{2}})^{\frac{2}{m}}\big](it^{\frac{m}{2}})^{-\alpha}t^{\frac{m}{2}}.
\ee
Taking upper limits on both sides of (\ref{6}) yields
\be\begin{split}\limsup_{t\rightarrow0^+}\frac{\omega(t)}{t^{\frac{m}{2}(\alpha-1)}}
&\leq
\limsup_{t\rightarrow0^+}\sum_{i> N}\exp\big[-(C_0-\epsilon)(it^{\frac{m}{2}})^{\frac{2}{m}}\big]
(it^{\frac{m}{2}})^{-\alpha}t^{\frac{m}{2}}\\
&\leq\int_0^{\infty}\exp\big[-(C_0-\epsilon)s^{\frac{2}{m}}\big]s^{-\alpha}ds<\infty.
\end{split}\nonumber\ee
Since $\epsilon$ is arbitrary, we have
\be\label{7}
\limsup_{t\rightarrow0+}\frac{\omega(t)}{t^{\frac{m}{2}(\alpha-1)}}\leq
\int_0^{\infty}\exp(-C_0s^{\frac{2}{m}})s^{-\alpha}ds<\infty.
\ee
This proves the second case in (\ref{omegacondition}).

\emph{Case $3$:} $\alpha=1$.
Let $\epsilon$ and $N$ be as above.
Let $t$ be sufficiently small and $t^{-m/2}>N$. Then
\be\label{12}\omega(t)=\sum_{i<\frac{1}{t^{m/2}}}e^{-\lambda_{i,p}t}\cdot i^{-1}+
\sum_{i\geq\frac{1}{t^{m/2}}}e^{-\lambda_{i,p}t}\cdot i^{-1}.\ee
The two terms on RHS of (\ref{12}) can be estimated as follows.
\be\label{13}\sum_{i<\frac{1}{t^{m/2}}}e^{-\lambda_{i,p}t}\cdot i^{-1}<
\sum_{i<\frac{1}{t^{m/2}}}i^{-1}<1+\int_{1}^{\frac{1}{t^{m/2}}}s^{-1}ds
=1+\frac{m}{2}\ln{\frac{1}{t}}\ee
and
\be
\begin{split}\sum_{i\geq\frac{1}{t^{m/2}}}e^{-\lambda_{i,p}t}\cdot i^{-1}&=
\sum_{it^{m/2}\geq1}\exp\big[-\frac{\lambda_{i,p}}{i^{\frac{2}{m}}}
(it^{\frac{m}{2}})^{\frac{2}{m}}\big](it^{\frac{m}{2}})^{-1}
t^{\frac{m}{2}}\\
&<\sum_{it^{m/2}\geq1}\exp\big[-(C_0-\epsilon)
(it^{\frac{m}{2}})^{\frac{2}{m}}\big](it^{\frac{m}{2}})^{-1}
t^{\frac{m}{2}}\qquad\big(\text{by (\ref{5}) and $t^{-m/2}>N$})\\
&\rightarrow \int_{1}^{\infty}\exp\big[-(C_0-\epsilon)s^{\frac{2}{m}}\big]s^{-1}ds<\infty, \qquad t\rightarrow 0^+.
\end{split}\nonumber\ee
Since $\epsilon$ is arbitrary, we have
\be\label{14}
\limsup_{t\rightarrow 0^+}\sum_{i\geq\frac{1}{t^{m/2}}}e^{-\lambda_{i,p}t}\cdot i^{-1}\leq  \int_{1}^{\infty}\exp(-C_0 s^{\frac{2}{m}})s^{-1}ds<\infty.\ee
Putting (\ref{12}), (\ref{13}) and (\ref{14}) together we have
$$
\limsup_{t\rightarrow0^+}\frac{\omega(t)}{\ln\frac{1}{t}}= \limsup_{t\rightarrow0^+}\frac{\sum_{i<\frac{1}{t^{m/2}}}e^{-\lambda_{i,p}t}\cdot i^{-1}}{\ln\frac{1}{t}}+\limsup_{t\rightarrow0^+}\frac{\sum_{i\geq\frac{1}{t^{m/2}}}e^{-\lambda_{i,p}t}\cdot i^{-1}}{\ln\frac{1}{t}}<\frac{m}{2},
$$
proving the third case in (\ref{omegacondition}).
\end{proof}

\end{document}